\documentclass[11pt, reqno]{amsart}

\usepackage{geometry}
\usepackage{graphicx} 
\usepackage{wrapfig}

\usepackage{amsmath,amssymb,amsfonts,amsthm}
\usepackage{mathtools}
\usepackage{bm}
\usepackage{bbm}
\usepackage{cite}
\usepackage{psfrag}
\usepackage{cite}
\usepackage{color,soul}
\usepackage{breakcites}     
\usepackage{bigints}
\usepackage{caption,subcaption}
\usepackage{enumerate}
\usepackage{enumitem}
\setlist{leftmargin=1.6em}

\usepackage[hidelinks]{hyperref} 

\newtheorem{theorem}{Theorem}
\newtheorem{proposition}{Proposition}
\newtheorem{assumption}{Assumption}
\newtheorem{lemma}{Lemma} 
\newtheorem{corollary}{Corollary}
\theoremstyle{definition}

\theoremstyle{remark}
\newtheorem{remark}{Remark}

\usepackage{graphicx}
\usepackage{xcolor}

\usepackage{algorithm}
\usepackage[noend]{algpseudocode}

\makeatletter
\def\subsubsection{\@startsection{subsubsection}{3}%
  \z@{.5\linespacing\@plus.7\linespacing}{-.5em}%
  {\normalfont\bfseries}}
\makeatother

\newcommand{\vz}[0]{\bm {z}}

\newcommand{\sign}[0]{\mathrm{sign}}

\newcommand{\calC}[0]{\mathcal{C}}

\newcommand{\calH}[0]{\mathcal{H}}

\newcommand{\calX}[0]{\mathcal{X}}

\newcommand{\calY}[0]{\mathcal{Y}}

\newcommand{\E}[0]{\mathbb{E}}

\newcommand{\R}[0]{\mathbb{R}}

\newcommand{\vertiii}[1]{{\left\vert\kern-0.25ex\left\vert\kern-0.25ex\left\vert #1 
    \right\vert\kern-0.25ex\right\vert\kern-0.25ex\right\vert}}

\newcommand{\vasti}{\bBigg@{3.5 }}
\newcommand{\vast}{\bBigg@{4}}
\newcommand{\Vast}{\bBigg@{5}}
\newcommand{\Vastt}{\bBigg@{7}}

\makeatother

\newcommand{\be}{\begin{equation}}
\newcommand{\ee}{\end{equation}}
\newcommand{\ba}{\begin{align}}
\newcommand{\ea}{\end{align}}
\newcommand{\baa}{\begin{align*}}
\newcommand{\eaa}{\end{align*}}

\newcommand{\rs}[1] {{\color{magenta}RS: [#1]}}

\newcommand{\Cipl}{C_{i,\eta}^+}
\newcommand{\Cimi}{C_{i,\eta}^-}
\newcommand{\Dij}{D_{ij}^{-\gamma}}

\newcommand{\Ai}{A_{i,\eta,\gamma}}
\newcommand{\Aij}{A_{i,j,\eta,\gamma}}

\newcommand{\Bi}{B_{i,\eta,\gamma}}

\newcommand{\Deps}{\Delta^\varepsilon}

\DeclareMathOperator{\diam}{diam}

\DeclareMathOperator{\dist}{dist}

\title{Approximation rates of entropic maps in semidiscrete optimal transport}

\thanks{Work done while Ritwik Sadhu was a graduate student at Cornell University, Department of Statistics and Data Science.
Z. Goldfeld is partially supported by NSF grants  CCF-2046018, and DMS-2210368, and the IBM Academic Award.
K. Kato is partially supported by NSF grants DMS-2210368 and DMS-2413405.}

\author[R. Sadhu]{Ritwik Sadhu}
\address[R. Sadhu]{Department of Statistics, University of Washington.}
\email{rsadhu@uw.edu}

\author[Z. Goldfeld]{Ziv Goldfeld}
\address[Z. Goldfeld]{School of Electrical and Computer Engineering, Cornell University.}
\email{goldfeld@cornell.edu}

\author[K. Kato]{Kengo Kato}
\address[K. Kato]{Department of Statistics and Data Science, Cornell University.}
\email{kk976@cornell.edu}

\begin{document}

\begin{abstract}
Entropic optimal transport offers a computationally tractable approximation to the classical problem. 
In this note, we study the approximation rate of the entropic optimal transport map (in approaching the Brenier map) when the regularization parameter $\varepsilon$ tends to zero in the semidiscrete setting, where the input measure is absolutely continuous while the output is finitely discrete. Previous work shows that the approximation rate is $O(\sqrt{\varepsilon})$ under the $L^2$-norm with respect to the input measure. 
In this work, we establish faster, $O(\varepsilon^2)$ rates up to polylogarithmic factors, under the dual Lipschitz norm, which is weaker than the $L^2$-norm. For the said dual norm, the $O(\varepsilon^2)$ rate is sharp. As a corollary, we derive a central limit theorem for the entropic estimator for the Brenier map in the dual Lipschitz space when the regularization parameter tends to zero as the sample size increases. 
\end{abstract}

\maketitle

\section{Introduction}

\subsection{Overview}
For an absolutely continuous input distribution $P$ and a generic output distribution $Q$, both on $\R^d$ with finite second moments, the \emph{Brenier map} \cite{brenier1991polar} sending $P$ to $Q$ induces the optimal coupling for the optimal transport problem with quadratic cost:
\begin{equation}
\inf_{\pi \in \Pi(P,Q)} \int \| x-y \|^2 \, d\pi(x,y), \label{eq: OT}
\end{equation}
where $\Pi(P,Q)$ denotes the collection of couplings of $P$ and $Q$.
The Brenier map can be characterized as a $P$-a.e. unique transport map given by the gradient of a convex function. This celebrated result has seen numerous applications in statistics and machine learning, ranging from transfer learning and domain adaptation to vector quantile regression and causal inference; see  \cite{chewi2024statistical} as an excellent review of the recent development in statistical optimal transport. From a mathematical standpoint, the Brenier map provides a powerful tool to derive functional inequalities \cite{cordero2002some} and suggests natural extensions of the quantile function to the multivariate setting \cite{chernozhukov2017monge}, among others. 

In practice, however, directly solving the optimal transport problem \eqref{eq: OT} and computing the Brenier map is challenging,  especially when $d$ is large. A popular remedy for this computational difficulty is entropic regularization, whereby \eqref{eq: OT} is replaced with 
\begin{equation}
\inf_{\pi \in \Pi(P,Q)} \int \| x-y \|^2 \, d\pi(x,y) + \varepsilon D_{\mathsf{KL}}(\pi \| P \otimes Q),
\label{eq: EOT1}
\end{equation}
where $\varepsilon > 0$ is the regularization parameter and $D_{\mathsf{KL}}$ is the Kullback-Leibler divergence defined by $D_{\mathsf{KL}}(\alpha \| \beta) \coloneqq \int \log \frac{d\alpha}{d\beta} \, d\alpha$ if $\alpha \ll \beta$ and $\coloneqq \infty$ otherwise. Entropic optimal transport is amenable to efficient computation via Sinkhorn's algorithm, for which rigorous convergence guarantees have been developed under different settings \cite{franklin1989scaling,cuturi2013lightspeed,altschuler2017near,peyre2019computational,berman2020sinkhorn,carlier2022linear,eckstein2022quantitative,ghosal2022convergence,nutz2023stability,conforti2023quantitative,chizat2024sharper}.  As $\varepsilon$ shrinks, various objects from entropic optimal transport converge to those for unregularized optimal transport---a topic that has seen extensive research activities in recent years; see the literature review below. 

Denoting by $\pi^\varepsilon$  the (unique) optimal coupling for the entropic problem \eqref{eq: EOT}, an entropic surrogate of the Brenier map is given by $T^\varepsilon(x) = \E_{(X,Y) \sim \pi^\varepsilon}[Y \mid X=x]$, which we shall call the \textit{entropic map}\cite{pooladian2021entropic}. To understand the quality of this computationally tractable approximation, the rate at which the entropic map approaches the Brenier map as $\varepsilon \downarrow 0$ has received recent attention. 
\cite{carlier2023convergence} showed that if $P$ and $Q$ are compactly supported and the Brenier map $T^0$ is $M$-Lipschitz (which precludes $Q$ being discrete), then $\| T^\varepsilon - T^0 \|_{L^2(P)}^2 \le M(d\varepsilon \log (1/\varepsilon) + O(\varepsilon))$.
In the continuous-to-continuous setting, imposing stronger smoothness conditions on the densities of $P$ and $Q$ and the dual potentials, \cite{pooladian2021entropic} established faster $O(\varepsilon^2)$ rates for $\| T^\varepsilon - T^0 \|_{L^2(P)}^2$. 
In the semidiscrete setting (i.e., when $P$ is absolutely continuous and $Q$ is finitely discrete), \cite{pooladian2023minimax} showed that 
\begin{equation}
\| T^\varepsilon - T^0 \|_{L^2(P)}^2 = O(\varepsilon),
\label{eq: L2 bound}
\end{equation}
and their Example 3.5 demonstrates that this rate is sharp under $L^2(P)$. The follow-up work by the same authors \cite{divol2024tight} derived quantitative upper bounds on the $L^2(P)$ error.

The goal of this paper is to explore quantitative upper bounds on the bias of $T^\varepsilon$ for small $\varepsilon$ in the semidiscrete setting, but from a different angle.
Instead of the $L^2$-norm, we shall look at the linear functional $\langle \varphi,T^\varepsilon \rangle_{L^2(P)}$ for a suitable Borel vector field $\varphi$ and derive quantitative upper bounds on $\langle \varphi,T^\varepsilon - T^0 \rangle_{L^2(P)}$.  The preceding bound \eqref{eq: L2 bound} by \cite{pooladian2023minimax} implies that, for any bounded Borel vector field $\varphi$,
\begin{equation}
| \langle \varphi, T^\varepsilon-T^0 \rangle_{L^2(P)}| \le \| \varphi \|_\infty  \| T^\varepsilon - T^0 \|_{L^2(P)} = O(\sqrt{\varepsilon}). 
\label{eq: slow rate}
\end{equation}
Somewhat surprisingly, this rate can be \textit{much} faster for smooth test functions. Indeed, our main result shows that, if $P$ is supported on a compact convex set and has a positive Lipschitz density on the support, then for any $\alpha$-H\"{o}lder vector field $\varphi$ with $\alpha \in (0,1]$,
\[
| \langle \varphi, T^\varepsilon-T^0 \rangle_{L^2(P)}| = O(\varepsilon^{1+\alpha} \vee \varepsilon^2 \log^3(1/\varepsilon)).
\]
In particular, this implies near $O(\varepsilon^2)$ approximation rates for Lipschitz test functions. The hidden constant depends on $\varphi$ only through its $\alpha$-H\"{o}lder norm, so by taking the supremum over $\varphi$ whose $\alpha$-H\"{o}lder norm is at most $1$, the same rate holds for $\| T^\varepsilon - T^0 \|_{(\calC^\alpha)^*}$, where $\| \cdot \|_{(\calC^\alpha)^*}$ is the dual norm. 
This fast convergence rate under the dual norm is in line with the (sharp) approximation rate of $\varepsilon^2$ for the semidiscrete optimal transportation cost itself \cite{altschuler2022asymptotics}. Finally, building on our recent work \cite{sadhu2023limit}, we derive a central limit theorem in the dual space $(\calC^\alpha)^*$  for the empirical entropic map with vanishing regularization parameters.

\subsection{Literature review}
There is now a large literature on convergence and approximation rates of entropic optimal transport costs, potentials, couplings, and maps when the regularization parameter tends to zero \cite{mikami2004monge,mikami2008optimal,leonard2012schrodinger,carlier2017convergence,chizat2020faster,pooladian2021entropic,conforti2021formula,altschuler2022asymptotics,nutz2022entropic,bernton2022entropic,delalande2022nearly,carlier2023convergence,pooladian2023minimax,pal2024difference,divol2024tight}. Among others,  \cite{altschuler2022asymptotics} derived an asymptotic expansion of the entropic cost in the semidiscrete case when the regularization parameter tends to zero, showing faster convergence at the rate $\varepsilon^2$ than the continuous-to-continuous case. Key to their derivation is the fact that the entropic dual potential vector ($z^\varepsilon$ below) converges toward the unregularized one with a rate faster than $\varepsilon$. The follow-up work by \cite{delalande2022nearly} establishes faster $O(\varepsilon^{1+\alpha'})$ rates for any $0 < \alpha' < \alpha$ for the entropic dual potential when the input density is $\alpha$-H\"{o}lder continuous with $\alpha \in (0,1]$.

There is also a growing interest in estimation and inference for the Brenier map \cite{chernozhukov2017monge,hutter2021minimax,ghosal2019multivariate,pooladian2021entropic,ghosal2022convergence,pooladian2022debiaser,divol2022optimal,sadhu2023limit,pooladian2023minimax,manole2021plugin}. Among them, \cite{pooladian2021entropic} proposed using the entropic map with vanishing regularization parameters to estimate the Brenier map, and established convergence rates under the $L^2(P)$-norm in the continuous-to-continuous setting. However, these rate are suboptimal from a minimax point of view \cite{hutter2021minimax}.
For the semidiscrete setting, \cite{pooladian2023minimax} established the $O(n^{-1/2})$ rate for the entropic estimator with vanishing regularization levels $\varepsilon = \varepsilon_n = O(n^{-1/2})$ under the squared $L^2(P)$-norm. Our recent work \cite{sadhu2023limit} derived various limiting distribution results for certain functionals of the empirical (unregularized) Brenier map, when the input $P$ is known but the discrete output $Q$ is unknown. 
Finally, while estimation of the Brenier map in the continuous-to-continuous setting suffers from the curse of dimensionality \cite{hutter2021minimax}, estimation of the entropic map with \textit{fixed} $\varepsilon$ enjoys parametric sample complexity, and several limiting distribution results have been derived; see \cite{del2023improved,rigollet2022sample,goldfeld2024limit}.

\subsection{Organization}
The rest of the note is organized as follows. 
Section \ref{sec: background} contains background material on the optimal transport problem  and its entropic counterpart. Section \ref{sec: main result} presents our main results. All the proofs are gathered in Section \ref{sec: proof}.

\subsection{Notation}
For $a,b \in \R$, we use the notation $a \vee b = \max \{a,b \}$ and $a \wedge b = \min \{ a,b \}$.
We use $\| \cdot \|$ and $\langle \cdot, \cdot \rangle$ to denote the Euclidean norm and inner product, respectively. Let $\mathbbm{1}_N \in \R^N$ denote the vector of ones. 
For $d \in \mathbb{N}$ and $0 \le r \le d$, $\calH^{r}$ denotes the $r$-dimensional Hausdorff measure on $\R^d$; cf. \cite{evans1991measure}.
\section{Background}
\label{sec: background}

\subsection{Optimal transport}
Let $P$ and $Q$ be Borel probability measures on $\R^d$ with finite second moments, and write $\calX$ and $\calY$ for their respective supports. Recall the quadratic optimal transport problem \eqref{eq: OT}, which, upon expanding the square, is equivalent~to 
\begin{equation}
\sup_{\pi \in \Pi(P,Q)}\int \langle x,y \rangle \, d\pi(x,y).  \label{eq: OT2}
\end{equation}
The Brenier theorem \cite{brenier1991polar} yields that whenever $P$ is absolutely continuous, the problem \eqref{eq: OT2} admits a unique optimal solution $\pi^0$, which is induced by a $P$-a.e. unique map $T^0: \calX \to \R^d$, in the sense that $\pi^0 = P \circ (\mathrm{id},T^0)^{-1}$ with $\mathrm{id}$ denoting the identity map.
We call $T^0$ the \textit{Brenier map}. 

The Brenier map can be characterized by the gradient of a convex potential solving the dual problem, which reads as
\[
\inf_{\substack{(\phi, \psi) \in L^1(P) \times L^1(Q) \\ \phi(x) + \psi (y) \ge \langle x,y \rangle, \forall (x,y) \in \calX \times \calY}} \int \phi \, dP + \int \psi \, dQ.
\]
One may replace $\phi$ with the convex conjugate of $\psi$, $\psi^* \coloneqq\sup_{y \in \calY} (\langle \cdot,y \rangle -\psi(y))$, which always satisfies the constraint and leads to the semidual problem
\[
\inf_{\psi \in L^1(Q)} \int \psi^* \, dP + \int \psi \, dQ.
\]
For any optimal solution $\psi$ to the semidual problem, the Brenier map is given by $T^0 (x)= \nabla \psi^*(x)$ for $P$-a.e. $x$. See, e.g., \cite{villani2008optimal,santambrogio15} for background of optimal transport.

\medskip 

We focus herein on the semidiscrete setting, where $P$ is absolutely continuous while $Q$ is finitely discrete with support $\calY = \{ y_1,\dots,y_N \}$. Let $q = (q_1,\dots,q_N)^\intercal$ be the vector of masses with $q_i = Q(\{ y_i \})$ for $i \in [N] \coloneqq \{ 1,\dots,N\}$. In this case, setting $z = (z_1,\dots,z_N)^\intercal$ with $z_i = \psi (y_i)$, the semidual problem reduces to 
\begin{equation}
\inf_{z \in \R^N} \int \max_{1 \le i \le N}(\langle x,y_i \rangle - z_i) \, dP(x) + \langle z,q \rangle. \label{eq: semidual}
\end{equation}
Given any $z^0 = (z_1^0,\dots,z_N^0)^\intercal$ optimal solution to \eqref{eq: semidual}, the Brenier map is given by 
\[
T^0 (x) = \nabla_{x} \Big (\max_{1 \le i \le N}(\langle x,y_i \rangle - z_i^0) \Big),\quad \mbox{$P$-a.e. $x$.}
\]
To simplify its description, for $z \in \R^N$, define the \textit{Laguerre cells} $\{ C_i(z) \}_{i=1}^N$,
\[
\begin{split}
C_i(z)
&\coloneqq \bigcap_{\substack{j \ne i; 1 \le j \le N}} \big \{ z \in \calX : \langle y_i-y_j,x \rangle \ge z_i -z_j \big \},
\end{split}
\]
using which the Brenier map is given by
\begin{equation}
T^0 (x) = y_i \quad \text{for} \ x \in  C_i(z^0) \ \text{and} \ i \in [N]. \label{eq: brenier}
\end{equation}
The Laguerre cells form a partition of $\calX$ up to Lebesgue negligible sets, so the description in \eqref{eq: brenier} specifies a $P$-a.e. defined map with values in $\calY$. Furthermore, as $T^0$ is a transport map, we have $P(C_i(z^0)) = Q(\{ y_i \}) = q_i > 0$
for $i \in [N]$. 

The dual vector $z^0$ is not unique as adding the same constant to all $z_i$ does not change the value of the objective in \eqref{eq: semidual}. So, we always normalize $z^0$ in such a way that $\langle z^0,\mathbbm{1}_N \rangle = 0$. 
Together with mild conditions on $P$, this normalization guarantees uniqueness of $z^0$.

\subsection{Entropic optimal transport} 
The entropic optimal transport problem corresponding to \eqref{eq: OT2} is
\begin{equation}
\sup_{\pi \in \Pi(P,Q)} \int \langle x,y \rangle \, d \pi(x,y) - \varepsilon D_{\mathsf{KL}}(\pi \| P \otimes Q),
\label{eq: EOT}
\end{equation}
where $\varepsilon > 0$ is the regularization parameter.
For any $P$ and $Q$ with finite second moments (i.e., beyond the semidiscrete setting), the problem \eqref{eq: EOT} admits a unique optimal solution $\pi^\varepsilon$, which is of the form 
\[
\frac{d\pi^\varepsilon}{d(P \otimes Q)} (x,y)= e^{\frac{\langle x,y \rangle - \phi^\varepsilon(x)-\psi^\varepsilon(y)}{\varepsilon}},
\]
where $(\phi^\varepsilon,\psi^\varepsilon)$ is any optimal solution to the dual problem\footnote{Pairs of optimal potentials are a.e. unique up to additive constants, i.e., if $(\tilde{\phi},\tilde{\psi})$ is another optimal pair then $\tilde{\phi} = \phi +c $ $P$-a.e. and $\tilde{\psi} = \psi  -c $ $Q$-a.e., for some $c\in\R$.}
\[
\inf_{(\phi,\psi) \in  L^1(P) \times L^1(Q)} \int \phi \, dP + \int \psi \, dQ + \varepsilon \iint e^{\frac{\langle x,y \rangle - \phi(x)-\psi(y)}{\varepsilon}} \, dP(x) dQ(y).
\]
Here, since $\pi^\varepsilon$ is a coupling, one has 
$
\int e^{\frac{\langle x,y \rangle - \phi^\varepsilon(x)-\psi^\varepsilon(y)}{\varepsilon}} dQ(y) = 1$,
that is,
\[
\phi^\varepsilon(x) = \varepsilon \log \int  e^{(\langle x,y \rangle - \psi^\varepsilon(y))/\varepsilon}\,  dQ(y),\quad \mbox{$P$-a.e. $x$.}
\]
Substituting this expression leads to the semidual problem
\[
\inf_{\psi \in L^1(Q)} \int \left \{ \varepsilon \log \int  e^{(\langle \cdot,y \rangle - \psi(y))/\varepsilon} \, dQ(y) \right \} \, dP + \int \psi \, dQ.
\]
See \cite{nutz2021introduction} for a comprehensive overview of entropic optimal transport. An entropic~counterpart of the Brenier map was proposed in \cite{pooladian2021entropic} by observing that $T^0(x) = \E_{(X,Y) \sim \pi^0}[Y \mid X=x]$, i.e., the Brenier map agrees with the conditional expectation of the second coordinate given the first under $\pi^0$. Replacing $\pi^0$ with $\pi^\varepsilon$ leads to the \textit{entropic map}
\[
T^\varepsilon (x) = \E_{(X,Y) \sim \pi^\varepsilon}[Y \mid X=x], \ x \in \calX. 
\]

Specializing to the semidiscrete setting where $Q$ has support $\calY = \{ y_1,\dots,y_N \}$, one may reduce the semidiscrete problem to 
\[
\inf_{z \in \R^N}\int \left \{ \varepsilon \log \sum_{i=1}^N q_i e^{(\langle \cdot,y_i \rangle - z_i)/\varepsilon}  \right \} \, dP + \langle z,q \rangle.
\]
Replacing $z_i$ with $z_i+\varepsilon \log q_i$, the above semidual problem is equivalent to 
\begin{equation}
\inf_{z \in \R^N}\int \left \{ \varepsilon \log \sum_{i=1}^N  e^{(\langle \cdot,y_i \rangle - z_i)/\varepsilon}  \right \} \, dP + \langle z,q \rangle. \label{eq: semidual EOT}
\end{equation}
By general theory of entropic optimal transport, the latter semidual problem \eqref{eq: semidual EOT} admits a unique optimal solution $z^\varepsilon$ subject to the normalization $\langle z^\varepsilon, \mathbbm{1}_N \rangle = 0$. Furthermore, the optimal coupling $\pi^\varepsilon$ is of the form
\[
\frac{d\pi^\varepsilon}{d(P \otimes R)} (x,y_i)= e^{\frac{\langle x,y_i \rangle - \phi^\varepsilon(x) - z_i^\varepsilon}{\varepsilon}}, \ x \in \calX , i \in [N],
\]
where $R$ is the counting measure on $\calY$ and
$\phi^\varepsilon(x) = \varepsilon \log \sum_{i=1}^N  e^{(\langle x,y_i \rangle - z_i^\varepsilon)/\varepsilon}$ for $x \in \calX$.
In this case, the entropic map further simplifies to 
\[
T^\varepsilon (x) = \sum_{i=1}^N y_i \frac{e^{(\langle x,y_i \rangle - z_i^\varepsilon)/\varepsilon}}{\sum_{j=1}^N e^{(\langle x,y_j \rangle - z_j^\varepsilon)/\varepsilon}}, \quad x \in \calX.
\]

\section{Main results}
\label{sec: main result}
We derive approximation rates of the entropic map $T^\varepsilon$ towards the Brenier map $T^0$ as $\varepsilon \downarrow 0$. In contrast to \cite{pooladian2023minimax,divol2024tight} that focus on the (squared) $L^2(P)$-norm $\| T^\varepsilon - T^0 \|_{L^2(P)}^2$, we consider the linear functional
\[
\langle \varphi, T^\varepsilon-T^0 \rangle_{L^2(P)} = \int \langle \varphi(x), T^\varepsilon (x)-T^0(x) \rangle \, dP(x),\]
for a suitable Borel vector field $\varphi: \calX \to \R^d$, and establish the rates. Taking the supremum over a certain function class leads to the convergence rates under the corresponding dual norm. We start from the assumption under which the results hold. 

\begin{assumption}[Conditions on marginals]
\label{asp: compact_convex_regular_density}
(i)  The input measure $P$ is supported on a compact convex set $\calX \subset \R^d$ with nonempty interior and has a Lebesgue density $\rho$ that is Lipschitz continuous and strictly positive on $\calX$.
(ii)   The output measure $Q$ is finitely discrete with support $\calY = \{ y_1,\dots,y_N \} \subset \R^d$. For $q = (q_1,\dots,q_N)^\intercal$ with $q_i = Q(\{ y_i \})$, we assume that $\min_{1 \le i \le N} q_i \ge c_0$ for some (sufficiently small) constant $c_0 \in (0,1)$. 
\end{assumption}

Condition (i) guarantees uniqueness of the dual vector $z^0$ (subject to the normalization $\langle z^0,\mathbbm{1}_N \rangle =0$); cf. Theorem 7.18 in \cite{santambrogio15}.
For a vector-valued mapping $\varphi: \calX \to \R^d$ and $\alpha \in (0,1]$, the $\alpha$-H\"{o}lder norm $\| \varphi \|_{\calC^\alpha}$ (Lipschitz norm when $\alpha=1$) is defined by
\[
\| \varphi \|_{\calC^\alpha} \coloneqq \| \varphi \|_{\infty} + \sup_{x,y \in\calX; x \ne y} \frac{\| \varphi(x) - \varphi(y)\|}{\|x-y\|^\alpha},
\]
where $\| \varphi \|_{\infty} = \sup_{x \in \calX}\| \varphi (x)\|$. The following is our main result. 

\begin{theorem}[Convergence rates for H\"{o}lder test functions]
\label{thm: main result}
Fix $\alpha \in (0,1]$.
Under Assumption~\ref{asp: compact_convex_regular_density}, for every $\alpha$-H\"{o}lder vector field $\varphi: \calX \to \R^d$, 
    \[
    |\langle \varphi, T^\varepsilon - T^0 \rangle_{L^2(P)}| \lesssim \| \varphi \|_{\infty}\varepsilon^2 \log^3(1/\varepsilon)   + \| \varphi \|_{\calC^\alpha}\varepsilon^{1+\alpha}, \quad \forall \varepsilon \in (0,1),  
    \]
    where the inequality $\lesssim$ holds up to a constant that depends only on $\alpha,\calX, \rho, \calY$, and $c_0$.
\end{theorem}

\begin{remark}[Bounded test functions]
Inspection of the proof shows that if the test function $\varphi$ is only (measurable and) bounded, then
$
    |\langle \varphi, T^\varepsilon - T^0 \rangle_{L^2(P)}| \lesssim \| \varphi \|_{\infty}\varepsilon$ for $\varepsilon \in (0,1)$,
    where the hidden constant depends only on $\calX, \rho, \calY$, and $c_0$. 
\end{remark}

Theorem \ref{thm: main result} implies that the (right) derivative of the mapping $\varepsilon \mapsto \langle \varphi, T^\varepsilon  \rangle_{L^2(P)}$ at $\varepsilon=0$ vanishes for any H\"{o}lder vector field $\varphi$. Indeed, the proof of the theorem shows that 
\[
\lim_{\varepsilon \downarrow 0} \frac{\langle \varphi, T^\varepsilon - T^0 \rangle_{L^2(P)}}{\varepsilon} = \sum_{i \ne j} \frac{\log 2}{\|y_i-y_j\|} \int_{C_i(z^0) \cap C_j(z^0)}\langle y_j-y_i,\varphi(x) \rangle \rho(x) \, d\calH^{d-1}(x),
\]
and the right-hand side vanishes.
Hence, we need to look at a higher-order expansion of the mapping $\varepsilon \mapsto \langle \varphi, T^\varepsilon  \rangle_{L^2(P)}$ around $\varepsilon=0$, which requires careful analysis of the facial structures of the Laguerre cells. In particular, special care is needed when $y_i-y_j$ and $y_i-y_k$ for some distinct indices $i,j,k$ are linearly dependent; see, e.g., the proof of Lemma~\ref{lem:slack_integral} ahead. The proof of Theorem \ref{thm: main result} is inspired by the proofs in \cite{altschuler2022asymptotics,delalande2022nearly} for the asymptotic expansions of the entropic cost, but differs from them in some important ways, as detailed in Remark \ref{rem: comparison} ahead. 

\begin{remark}[Sharpness of $O(\varepsilon^2)$ rate when $\alpha=1$]
Let $W_2^2(P,Q)$ denote the squared $2$-Wasserstein distance, i.e., the optimal value in \eqref{eq: OT}. Theorem 1.1 in \cite{altschuler2022asymptotics} establishes 
\[
\E_{(X,Y) \sim \pi^\varepsilon}[\|X-Y\|^2] = W_2^2(P,Q) + \frac{\varepsilon^2\pi^2}{12} \sum_{i < j} \frac{1}{\|y_i-y_j\|}\int_{C_i(z^0) \cap C_j(z^0)} \rho (x) \, d\calH^{d-1}(x) + o(\varepsilon^2).
\]
Rearranging terms, this implies that
\[
\langle \mathrm{id}, T^\varepsilon - T^0 \rangle_{L^2(P)} = -\frac{\varepsilon^2\pi^2}{24} \sum_{i < j} \frac{1}{\|y_i-y_j\|}\int_{C_i(z^0) \cap C_j(z^0)} \rho (x) \, d\calH^{d-1}(x) + o(\varepsilon^2).
\]
Since the identity mapping $\mathrm{id}$ is Lipschitz, the rate in Theorem \ref{thm: main result} is sharp up to the  $\log^3(1/\varepsilon)$ factor. The question of whether the polylogarithmic factor can be dropped for a generic Lipschitz vector is left for future research. 
\end{remark}

\begin{remark}[Sharpness of $O(\varepsilon^{\alpha+1})$ rate in $d=1$]
As in \cite{altschuler2022asymptotics,pooladian2023minimax}, consider $d=1, P=\mathrm{Unif} ([-1,1])$, and $Q = \frac{1}{2}(\delta_{-1}+\delta_{1})$, for which the entropic map is $T^\varepsilon (x) = \tanh (2x/\varepsilon)$ and the Brenier map is $T^0 (x) = \sign (x)$. For $\varphi(x) = \sign (x) |x|^{\alpha}$ with $\alpha \in (0,1]$, which is $\alpha$-H\"{o}lder on $[-1,1]$, one can verify from the dominated convergence theorem that 
\[
\lim_{\varepsilon \downarrow 0} \varepsilon^{-1-\alpha} \langle \varphi,T^{\varepsilon} - T^0 \rangle_{L^2(P)} = \int_0^\infty x^{\alpha} (\tanh (2x) - 1) \, dx,
\]
where the integral on the right-hand side is absolutely convergent. Hence, the $O(\varepsilon^{1+\alpha})$ rate in Theorem \ref{thm: main result} is in general sharp for $\alpha \in (0,1)$.
\end{remark}

Let $\calC^\alpha = \calC^\alpha(\calX;\R^d)$ be the Banach space of $\alpha$-H\"older mappings $\calX \to \R^d$ endowed with the norm $\| \cdot \|_{\calC^\alpha}$. The topological dual $(\calC^\alpha)^*$ is the Banach space of continuous linear functionals on $\calC^\alpha$ endowed with the dual norm, $\| \ell \|_{(\calC^\alpha)^*} = \sup_{\varphi: \| \varphi \|_{\calC^\alpha} \le 1} \ell (\varphi)$.  One may think of any bounded measurable mapping $T: \calX \to \R^d$ as an element of the dual space $(\calC^\alpha)^*$ by identifying $T$ with the linear functional $\varphi \mapsto \langle \varphi, T \rangle_{L^2(P)}$. 
With this identification, the preceding theorem yields rates of convergence of the entropic map under  $\| \cdot \|_{(\calC^\alpha)^*}$.
\begin{corollary}[Convergence rates under dual H\"{o}lder norm]
\label{cor: dual norm}
Fix $\alpha \in (0,1]$. 
Under Assumption~\ref{asp: compact_convex_regular_density}, 
    \[
    \| T^\varepsilon - T^0\|_{(\calC^\alpha)^*}\lesssim \varepsilon^{1+\alpha} \vee \varepsilon^2 \log^3(1/\varepsilon), \quad \forall \varepsilon \in (0,1),  
    \]
    where the inequality $\lesssim$ holds up to a constant that depends only on $\alpha, \calX, \rho, \calY$, and $c_0$.
\end{corollary}

We discuss a statistical application of the preceding result. Suppose the input measure $P$ is known but the output $Q$ is unknown, and we have access to an i.i.d. sample $Y_1,\dots,Y_n$ from $Q$. Such a setting is natural when we think of the Brenier map as a multivariate quantile function, where $P$ serves as a reference measure (cf. \cite{chernozhukov2017monge}). Let $\hat{Q}_n = n^{-1}\sum_{i=1}^n\delta_{Y_i}$ denote the empirical distribution, which is supported in $\calY$. In addition, let $\hat{T}_n^0$ and $\hat{T}_n^\varepsilon$ with $\varepsilon > 0$ be the Brenier and entropic maps, respectively, for the pair $(P,\hat{Q}_n)$.
Our recent work \cite{sadhu2023limit} established a central limit theorem for $\hat{T}_n^0$ in $(\calC^\alpha)^*$, 
\begin{equation}
\sqrt{n}(\hat{T}_n^0 - T^0) \stackrel{d}{\to} \mathbb{G} \quad \text{in} \ (\calC^\alpha)^*, \quad \text{as}  \
 n \to \infty,
\label{eq: clt}
\end{equation}
where $\stackrel{d}{\to}$ signifies convergence in distribution and $\mathbb{G}$ is a centered Gaussian variable in $(\calC^\alpha)^*$ (the exact form of $\mathbb{G}$ can be found in Theorem 4 in \cite{sadhu2023limit}). The next result shows that the same weak limit holds for the entropic estimator with $\varepsilon = \varepsilon_n \downarrow 0$ sufficiently fast. 

\begin{corollary}[Central limit theorem under dual H\"{o}lder space]
\label{cor: clt}
Suppose Assumption~\ref{asp: compact_convex_regular_density} holds and in addition that one of the following holds for $\calX$: (a) $\calX$ is a polytope, or (b) $\calH^{d-1}(\partial \calX \cap H) = 0$ for every hyperplane $H$ in $\R^d$. Then, 
\[
\sqrt{n}(\hat{T}_n^{\varepsilon_n} - T^0) \stackrel{d}{\to} \mathbb{G}\quad \text{in} \ (\calC^\alpha)^*,
\]
provided that $\varepsilon_n = o\big ( n^{-\frac{1}{2(1+\alpha)}} \wedge n^{-1/4}/\log^{3/2}n \big)$, where $\mathbb{G}$ is the same centered Gaussian variable in  $(\calC^\alpha)^*$ as that in \eqref{eq: clt}. 
\end{corollary}

\begin{remark}[Comparison with \cite{pooladian2023minimax}]
\cite{pooladian2023minimax} showed that $\E[\| \hat{T}_n^{\varepsilon_n} - T^{\varepsilon_n}\|_{L^2(P)}^2] = O(\varepsilon_n^{-1}n^{-1})$. Combining the bias estimate in \eqref{eq: L2 bound}, they established $\E[\| \hat{T}_n^{\varepsilon_n} - T^{0}\|_{L^2(P)}^2] = O(n^{-1/2})$ by choosing $\varepsilon_n$ decaying at the rate $n^{-1/2}$. It is interesting to observe that, under the dual norm $\| \cdot \|_{(\calC^\alpha)^*}$, the empirical entropic map enjoys the parametric rate with $\varepsilon_n$ decaying substantially slower than $n^{-1/2}$.
\end{remark}

\section{Proofs}
\label{sec: proof}
\subsection{Preliminaries}

Define
\[
\Delta^{\varepsilon}_{ij}(x)  \coloneqq \langle y_i - y_j, x \rangle - z_i^\varepsilon + z_j^\varepsilon, \quad \varepsilon \ge 0.
\]
Observe that $C_i(z^0) = \{ x \in \calX : \Delta_{ij}^0 \ge 0, \forall j \ne i \}$ and 
\[
T^\varepsilon (x)=\sum_{j=1}^N y_j \frac{e^{-\Delta_{ij}^\varepsilon(x)/\varepsilon}}{\sum_{k=1}^N e^{-\Delta_{ik}^\varepsilon(x)/\varepsilon}}
\]
for any $i \in [N]$ and $x \in \calX$.
Furthermore, define
\[
 H_{ij}(t) \coloneqq \{ x \in C_i(z^0): \Delta_{ij}^0(x) = t\}.
\]
Observe that $H_{ij} (0) = H_{ji}(0) = C_i(z^0) \cap C_j(z^0)$.
For notational convenience, set $M_\rho = \sup_{x \in \calX} \rho(x) < \infty$ and $\delta=\min_{i \ne j}\|y_i-y_j\| > 0$. In what follows, the notation $\lesssim$ means that the left-hand side is upper bounded by the right-hand side up to a constant that depends only on $\alpha, \calX,\rho, \calY$, and $c_0$. We first establish the following preliminary estimates.

\begin{lemma}
\label{lem:slack_integral}
    Under Assumption~\ref{asp: compact_convex_regular_density}, the following hold. 
    \begin{enumerate}
    \item[(i)] For any distinct indices $i,j$, one has 
    $\int_{C_i(z^0)}  e^{-\Delta_{ij}^0(x)/\varepsilon} \rho(x) \, dx \le \frac{\varepsilon M_\rho (\diam \calX)^{d-1}}{\|y_i-y_j\|}.
   $
   \item[(ii)] For any distinct indices $i,j,k$,
    \[
\int_{C_i(z^0)} e^{-\Delta^0_{ij}(x)/\varepsilon} e^{-\Delta^0_{ik}(x)/\varepsilon} \, \rho(x) \, dx \lesssim \varepsilon^2\log^2(1/\varepsilon), \quad \forall \varepsilon > 0.
\]
    \end{enumerate} 
\end{lemma}

\begin{proof}[Proof of Lemma \ref{lem:slack_integral}]
(i). By the coarea formula \cite[Theorem 3.11]{evans1991measure}, 
\begin{equation}
    \int_{C_i(z^0)}e^{-\Delta_{ij}^0(x)/\varepsilon} \rho(x) \, dx = \frac{1}{\| y_i-y_j \|} \int_0^\infty \left ( \int_{H_{ij}(t)} \rho(x) \, d\calH^{d-1}(x) \right )e^{-t/\varepsilon} dt. 
    \label{eq: slack}
\end{equation}
The inner integral can be bounded by
$
M_\rho\calH^{d-1}\big( H_{ij}(t) \big ) \le M_\rho (\diam \calX)^{d-1}, 
$
as $H_{ij}(t)$ is a hyperplane section of $\calX$, which implies that the right-hand side on \eqref{eq: slack} can be bounded by $\varepsilon \|y_i-y_j\|^{-1} M_\rho (\diam \calX)^{d-1}$.

\medskip
(ii).  Fix $\eta > 0$. Set
\[
A_{i\ell}(\eta) \coloneqq \{x \in C_i(z^0): \Delta^0_{i\ell}(x) \geq \eta\} \quad \text{and} \quad
B_{i\ell}(\eta) \coloneqq \{x \in \calX: 0 \leq \Delta^0_{i\ell}(x) < \eta\},
\]
for $\ell = j,k$.  Then, applying the coarea formula, one has
\[
\begin{split}
    &\int_{C_i(z^0)} e^{-\Delta^0_{ij}(x)/\varepsilon} e^{-\Delta^0_{ik}(x)/\varepsilon} \rho(x)\,dx \\
    &\leq \left ( \int_{A_{ij}(\eta)} + \int_{A_{ik}(\eta)} + \int_{C_i(z^0) \cap A_{ij}(\eta)^c \cap A_{ik}(\eta)^c} \right ) e^{-\Delta^0_{ij}(x)/\varepsilon} e^{-\Delta^0_{ik}(x)/\varepsilon} \rho(x)\,dx\\
    &\leq \delta^{-1} e^{-\eta/\varepsilon}  \int_0^\infty \left \{\left( \int_{H_{ij}(t)} + \int_{H_{ik}(t)} \right )\rho(x) \, d\calH^{d-1}(x)  \right \} e^{-t/\varepsilon} \,dt 
 +  M_\rho \calH^d(B_{ij}(\eta) \cap B_{ik}(\eta))\\
    &\leq 2 \delta^{-1}\varepsilon e^{-\eta/\varepsilon} (\diam \calX)^{d-1} M_\rho + M_\rho \calH^d(B_{ij}(\eta) \cap B_{ik}(\eta)).
\end{split}
\]
For the second term on the right-hand side, we separately consider the following two cases. 

\medskip

Case (a). Suppose that  $y_i - y_j$ and $y_i - y_k$ are linearly independent.  In this case, 
\[
\calH^d(B_{ij}(\eta) \cap B_{ik}(\eta)) \leq (\diam \calX)^{d-2} \frac{\eta^2}{\sqrt{\|y_i - y_j\|^2 \|y_i - y_k\|^2 - \langle y_i - y_j, y_i - y_k \rangle^2}}.
\]

Case (b). Suppose that $y_i - y_j$ and $y_i - y_k$ are linearly dependent, so that $y_i - y_k = c(y_i - y_j)$ for some $c \ne 0$. We will show that there exists $\eta_0 > 0$ that depends only on $\calX, \rho, \calY$, and $c_0$ such that $B_{ij}(\eta) \cap B_{ik}(\eta) = \varnothing$ for all $\eta \in (0,\eta_0)$. We only consider the $c < 0$ case. The $c > 0$ case is similar (see Step 1 of the proof of Theorem 1~(i) in \cite{sadhu2023limit} for a similar argument). Suppose $B_{ij}(\eta) \cap B_{ik}(\eta) \ne \varnothing$, which entails that there exists some $x \in \calX$ such that 
\begin{equation}
0 \le \langle y_i-y_j,x \rangle - b_{ij} < \eta \quad \text{and} \quad 0 \le \langle y_i-y_k,x \rangle - b_{ik} < \eta,  
\label{eq: hyperplane}
\end{equation}
where $b_{ij} = z_i^0 - z_j^0$. Let $L_1$ and $L_2$ be the hyperplanes defined by $L_1 = \{ x : \langle y_i-y_j,x \rangle  = b_{ij} \}$ and $L_2 = \{  x: \langle y_i-y_k,x \rangle  = b_{ik} \}$, which are parallel as $y_i-y_j$ and $y_j-y_k$ are linearly dependent. As such,
\[
\dist (L_1,L_2) = \frac{|b_{ij} - c^{-1}b_{ik}|}{\|y_i-y_j\|}. 
\]
On the other hand, by our choice of $x$ from \eqref{eq: hyperplane}, 
\[
\dist (L_1,L_2) \le \dist (x,L_1) + \dist (x,L_2) \le \frac{\eta}{\|y_i-y_j\|} + \frac{\eta}{\|y_i-y_k\|} = \frac{\eta (1+|c|^{-1})}{\| y_i-y_j \|},
\]
so that $|b_{ij} - c^{-1}b_{ik}| \le \eta (1+|c|^{-1})$. Observe that
\[
\begin{split}
    C_i(z^0) &\subset \{x: \langle y_i - y_j, x \rangle \geq b_{ij} \} \cap \{ x: \langle y_i - y_k, x \rangle \geq b_{ik} \}\\
    &= \{x: \langle y_i - y_j, x \rangle \geq b_{ij} \} \cap \{ x: \langle y_i - y_j, x \rangle \leq  c^{-1}b_{ik} \}\\
    &\subset \{x: \langle y_i - y_j, x \rangle \geq b_{ij} \} \cap \{ x: \langle y_i - y_j, x \rangle \leq  b_{ij} +  \eta (1+|c|^{-1}) \},
\end{split}
\]
which implies that 
\[
\begin{split}
q_i = P(C_i(z^0)) \leq M_\rho (\diam \calX)^{d-1} \frac{\eta (1+|c|^{-1})}{\|y_i - y_j\|}.
\end{split}
\]
Hence, if we choose
\[
\eta_0 = \frac{\delta c_0}{2(1+|c|^{-1}) M_\rho (\diam \calX)^{d-1}},
\]
then $q_i < c_0 \le \min_{\ell}q_\ell$ for $\eta < \eta_0$, which is a contradiction. Conclude that $B_{ij} (\eta) \cap B_{ik}(\eta) = \varnothing$ for $\eta < \eta_0$.
\medskip

Finally, by choosing $\eta = \varepsilon \log(1/\varepsilon)$, we see that the desired estimate holds for all $\varepsilon \in (0,\varepsilon_0)$ for some $\varepsilon_0 > 0$ that depends only on $\calX, \rho, \calY$, and $c_0$. For $\varepsilon \ge \varepsilon_0$, one may use the crude upper bound 
\[
\int_{C_i(z^0)}  e^{-\Delta_{ij}^0(x)/\varepsilon} e^{-\Delta_{ik}^0(x)/\varepsilon}\rho(x) \, dx \le \int_{C_i(z^0)} \rho (x) \, dx \le 1,
\]
and adjust the constant hidden in $\lesssim$. 
\end{proof}

\subsection{Proof of Theorem \ref{thm: main result}} The proof is divided into two steps.

\medskip

\textbf{Step 1}. We first establish that 
\begin{equation}
|\langle \varphi, T^\varepsilon - T^0 \rangle_{L^2(P)} | \lesssim \| \varphi \|_{\infty} \big( \|z^\varepsilon-z^0\|_{\infty} e^{2\|z^\varepsilon-z^0\|_{\infty}/\varepsilon} + \varepsilon^2 \log^2(1/\varepsilon) \big) + \| \varphi \|_{\calC^\alpha}\varepsilon^{1+\alpha}. 
\label{eq: step1}
\end{equation}
Since $\{ C_i(z^0) \}_{i=1}^N$ forms a partition of $\calX$ up to Lebesgue negligible sets, one has
\[
\begin{split}
\langle \varphi, T^\varepsilon \rangle_{L^2(P)} 
=
\sum_{i=1}^N \sum_{j=1}^N \int_{C_i(z^0)}\langle y_j, \varphi(x) \rangle \frac{e^{-\Delta_{ij}^\varepsilon(x)/\varepsilon}}{\sum_{k=1}^N e^{-\Delta_{ik}^\varepsilon(x)/\varepsilon}} \rho(x) \, dx.
\end{split}
\]
On the other hand, 
\[
\langle \varphi, T^0 \rangle_{L^2(P)} = \sum_{i=1}^N \sum_{j=1}^N \int_{C_i(z^0)}\langle y_i, \varphi(x) \rangle \frac{e^{-\Delta_{ij}^\varepsilon(x)/\varepsilon}}{\sum_{k=1}^N e^{-\Delta_{ik}^\varepsilon(x)/\varepsilon}} \rho(x) \, dx.
\]
Subtracting these expressions leads to 
\begin{equation}
\begin{split}
\langle \varphi, T^\varepsilon - T^0 \rangle_{L^2(P)} 
=\sum_{i \ne j} \int_{C_i(z^0)}\langle y_j-y_i, \varphi(x) \rangle \frac{e^{-\Delta_{ij}^\varepsilon(x)/\varepsilon}}{1+\sum_{k \ne j} e^{-\Delta_{ik}^\varepsilon(x)/\varepsilon}} \rho(x) \, dx.
\end{split}
\label{eq: integrand}
\end{equation}
We will replace $\Delta_{ij}^\varepsilon$ with $\Delta_{ij}^0$ on the right-hand side. 

Noting that $e^{-\Delta_{ij}^\varepsilon/\varepsilon} = e^{(z_i^\varepsilon-z_i^0-z_j^\varepsilon+z_j^0)/\varepsilon}e^{-\Delta_{ij}^0/\varepsilon}$ and $\Delta_{ik}^0 \ge 0$ for $k \ne i$ on $C_i(z^0)$ and using the elementary inequality $|e^t-1| \le e^{|t|}|t|$, one has, for $x \in C_i(z^0)$,
\begin{align*}
&\left | \frac{e^{-\Delta_{ij}^\varepsilon(x)/\varepsilon}}{1+\sum_{k \ne j} e^{-\Delta_{ik}^\varepsilon(x)/\varepsilon}} - \frac{e^{-\Delta_{ij}^0(x)/\varepsilon}}{1+\sum_{k \ne j} e^{-\Delta_{ik}^0(x)/\varepsilon}} \right | \\
&\le\left| e^{-\Delta_{ij}^\varepsilon(x)/\varepsilon} \Big(1+ \sum_{k \ne i} e^{-\Delta_{ik}^0(x)/\varepsilon} \Big) - e^{-\Delta_{ij}^0(x)/\varepsilon} \Big(1+ \sum_{k \ne i} e^{-\Delta_{ik}^\varepsilon(x)/\varepsilon} \Big) \right |\\
&= \Bigg| e^{-\Delta_{ij}^0(x)/\varepsilon}\big(e^{(z_i^\varepsilon-z_i^0-z_j^\varepsilon+z_j^0)/\varepsilon}-1+1\big) \Big(1+ \sum_{k \ne i} e^{-\Delta_{ik}^0(x)/\varepsilon} \Big) \\
&\qquad \qquad - e^{-\Delta_{ij}^0(x)/\varepsilon} \Big(1+ \sum_{k \ne i} e^{-\Delta_{ik}^0(x)/\varepsilon} \big(e^{(z_i^\varepsilon-z_i^0-z_k^\varepsilon+z_k^0)/\varepsilon}-1+1\big)\Big) \Bigg | \\
&\le 4\varepsilon^{-1}N \|z^\varepsilon-z^0\|_{\infty}e^{-\Delta_{ij}^0(x)/\varepsilon} e^{2\|z^\varepsilon-z^0\|_{\infty}/\varepsilon}. 
\end{align*}
Lemma \ref{lem:slack_integral}~(i) then yields
\[
\begin{split}
&\left | \langle \varphi, T^\varepsilon - T^0 \rangle_{L^2(P)} - \sum_{i \ne j} \int_{C_i(z^0)}\langle y_j-y_i, \varphi(x) \rangle \frac{e^{-\Delta_{ij}^0(x)/\varepsilon}}{1+\sum_{k \ne j} e^{-\Delta_{ik}^0(x)/\varepsilon}} \rho(x) \, dx\right | \\
&\le 4N^3 \| \varphi \|_{\infty}   M_\rho (\diam \calX)^{d-1} \|z^\varepsilon-z^0\|_{\infty} e^{2\|z^\varepsilon-z^0\|_{\infty}/\varepsilon}. 
\end{split}
\]

Furthermore, 
\[
\left |\frac{e^{-\Delta_{ij}^0(x)/\varepsilon}}{1+\sum_{k \ne j} e^{-\Delta_{ik}^0(x)/\varepsilon}} - \frac{e^{-\Delta_{ij}^0(x)/\varepsilon}}{1+e^{-\Delta_{ij}^0(x)/\varepsilon}} \right | \le e^{-\Delta_{ij}^0(x)/\varepsilon}\sum_{k \ne i,j}e^{-\Delta_{ik}^0(x)/\varepsilon}.
\]
Hence, by Lemma \ref{lem:slack_integral}~(ii), we conclude that
\[
\begin{split}
&\left | \langle \varphi, T^\varepsilon - T^0 \rangle_{L^2(P)} - \sum_{i \ne j} \int_{C_i(z^0)}\langle y_j-y_i, \varphi(x) \rangle \frac{e^{-\Delta_{ij}^0(x)/\varepsilon}}{1+e^{-\Delta_{ij}^0(x)/\varepsilon}} \rho(x) \, dx\right | \\
&\lesssim \| \varphi \|_{\infty} \big ( \|z^\varepsilon-z^0\|_{\infty} e^{2\|z^\varepsilon-z^0\|_{\infty}/\varepsilon} + \varepsilon^2 \log^2(1/\varepsilon) \big). 
\end{split}
\]
Setting 
\[
h_{ij}^\varphi(t) = \int_{H_{ij}(t)} \langle y_j-y_i,\varphi(x) \rangle \rho(x) \, d\calH^{d-1}(x), 
\]
an application of the coarea formula yields
\[
\int_{C_i(z^0)} \langle y_j-y_i, \varphi(x) \rangle \frac{e^{-\Delta_{ij}^0(x)/\varepsilon}}{1+e^{-\Delta_{ij}^0(x)/\varepsilon}} \rho(x) \, dx = \frac{\varepsilon}{\| y_i-y_j \|} \int_{0}^\infty h_{ij}^\varphi(\varepsilon t) \frac{e^{-t}}{1+e^{-t}} \, dt.
\]
We will replace $h_{ij}^\varphi(\varepsilon t)$ with $h_{ij}^\varphi(0)$. To this end, we need the following estimate, whose proof will be given after the proof of this theorem.  
\begin{lemma}
\label{lem: hyperplane}
For any distinct indices $i,j$, $\calH^{d-1} \big (H_{ij}(t) \Delta [H_{ij}(0) + tv_{ij}] \big) \lesssim t$ for all $t > 0$ with $v_{ij} = (y_i-y_j)/\|y_i-y_j\|^2$. Here $[H_{ij}(0)+tv_{ij}] = \{ x + tv_{ij} : x \in H_{ij}(0) \}$.
\end{lemma}

The above lemma yields 
\[
|h_{ij}^\varphi (t) - h_{ij}(0)| \lesssim \| \varphi \|_{\infty}t + \int_{H_{ij}(0)} \| \varphi (x+tv_{ij}) \rho(x+tv_{ij}) - \varphi (x) \rho(x) \| \, d\calH^{d-1}(x) \lesssim \| \varphi \|_{\calC^\alpha} (t \vee t^\alpha), 
\]
where we used the fact that $\rho$ is Lipschitz and $\calX$ is bounded. This implies 
\[
\left | \int_{0}^\infty (h_{ij}^\varphi(\varepsilon t)-h_{ij}^\varphi(0)) \frac{e^{-t}}{1+e^{-t}} \, dt \right |  \lesssim \| \varphi \|_{\calC^\alpha}  \varepsilon^\alpha, \quad \varepsilon \in (0,1),
\]
Furthermore, since $h^\varphi_{ij} (0) = - h_{ji}^\varphi(0)$ (as $H_{ij}(0) = H_{ji}(0) = C_i(z^0) \cap C_j(z^0)$), we have
$\sum_{i \ne j} h^\varphi_{ij}(0)/\|y_i-y_j\|=0$.
Putting everything together, we obtain the estimate in \eqref{eq: step1}. 

\medskip

\textbf{Step 2}. In this step, we establish that 
$\| z^\varepsilon - z^0 \| \lesssim \varepsilon^2 \log^3(1/\varepsilon)$,
which, combined with Step 1, leads to the result of the theorem. This is a slight improvement on Corollary 2.2 in \cite{delalande2022nearly}, but follows from the arguments there with a minor modification. We provide an outline below.

Set
\[
G_i(\varepsilon, z) = \int \frac{e^{(\langle x,y_i \rangle - z_i)/\varepsilon}}{\sum_{j=1}^N e^{(\langle x,y_j \rangle - z_j)/\varepsilon}} \rho(x) \, dx - q_i,\,i \in [N],
\]
and $G(\varepsilon, z) = (G_1(\varepsilon,z), \dots, G_N(\varepsilon, z))^\intercal$. By the first-order condition for the semidual problem \eqref{eq: semidual EOT}, $z^\varepsilon$ for $\varepsilon > 0$ satisfies 
$G(\varepsilon,z^\varepsilon) = 0$. 
By Theorem 3.2 in \cite{delalande2022nearly}, $\nabla_{z}G(\varepsilon,z^\varepsilon)$ is invertible on $(\mathbbm{1}_N)^\bot$ (the vector subspace of $\R^N$ orthogonal to $\mathbbm{1}_N$), so the implicit function theorem yields that the mapping $\varepsilon \mapsto z^\varepsilon$ is $\calC^1$ on $(0,\infty)$ with
$\dot{z}^\varepsilon = -\big[ \nabla_z G(\varepsilon,z^\varepsilon) \big ]^{-1} \dot{G}(\varepsilon,z^\varepsilon)$, 
where $\dot{z}^\varepsilon = dz^\varepsilon/d\varepsilon$ and $\dot G(\varepsilon,z) = \partial G(\varepsilon,z)/\partial \varepsilon$ (note here that $\dot G(\varepsilon,z) \in (\mathbbm{1}_N)^\bot$). Again, using Theorem 3.2 in \cite{delalande2022nearly}, one obtains 
$\| \dot{z}^\varepsilon \| \lesssim \| \dot{G}(\varepsilon,z^\varepsilon) \|/\lambda_2$, 
where $\lambda_2$ denotes the second smallest eigenvalue of the covariance matrix of $Q$. By \cite{tanabe1992exact}, $\lambda_2 \gtrsim 1$. Finally, the proof of Theorem 3.3 in \cite{delalande2022nearly} yields that for any $\eta > 0$, 
\[
|\dot{G}_i(\varepsilon,z^\varepsilon)| \lesssim \frac{\eta^3}{\varepsilon^2} + \frac{e^{-\eta/\varepsilon}}{\varepsilon^2} \big ( 1+\eta^2 + \varepsilon \eta + (\eta+\varepsilon^2)e^{-\eta/\varepsilon} \big), \ i \in [N].
\]
Choosing $\eta = 3\varepsilon \log (1/\varepsilon)$ leads to $\| \dot{z}^\varepsilon \| \lesssim \varepsilon \log^3(1/\varepsilon)$, so that $\| z^\varepsilon - z^0 \| \le \int_0^\varepsilon \| \dot{z}^t \| \, dt \lesssim \varepsilon^2 \log^3(1/\varepsilon)$. This completes the proof. \qed

\begin{proof}[Proof of Lemma \ref{lem: hyperplane}]
Set $b_{ij} = z_i^0-z_j^0$ for notational convenience. 
Since $x \in [H_{ij}(0) + t v_{ij}]$ for $t > 0$ satisfies $
\langle y_i-y_j,x \rangle - b_{ij}= t$, one sees that $[H_{ij}(0) + t v_{ij}] \setminus H_{ij}(t) \subset C_i(z^0)^c \cap C_j(z^0)^c$.
Set $H_{ijk}(t) = \{x: x \in H_{ij}(0), x + t v_{ij} \in C_k(z^0)\}$, then $
 [H_{ij}(0) + t v_{ij}] \setminus H_{ij}(t)   \subset \bigcup_{k \ne i,j} \big[H_{ijk}(t) + t v_{ij}\big]$.
For $x \in H_{ijk}(t)$, the translation of $x$ by $tv_{ij}$ alters the sign of   $\langle y_{i} - y_k, x \rangle - b_{ik}$, which can happen only when
$
0 \leq \langle y_i - y_k, x \rangle - b_{ik} \leq t\|y_i - y_k\|/\|y_i - y_j\|$. 
This implies $H_{ijk}(t) \subset 
\big\{ x \in \calX :\langle y_i - y_j, x \rangle = b_{ij},  b_{ik} \leq \langle y_i - y_k, x \rangle  \leq b_{ik} + R_{\calY} t \big\} =: A_{ijk}(t)  
$
with $R_\calY = \max_{\text{$i,j,k$ distinct}}\frac{\|y_i - y_k\|}{\|y_i - y_j\|}$. 
We separately consider the following two cases.

\medskip

Case (i). 
Suppose that $y_i - y_j$ and $y_i - y_k$ are linearly independent. In this case $\calH^{d-1}(A_{ijk}(t))\lesssim t$. 

\medskip

Case (ii). Suppose that $y_i - y_j$ and $y_i - y_k$ are linearly dependent, i.e., $y_i - y_k = c (y_i-y_j)$ for some $c\neq 0$. 
Set $L_1 = \{x : \langle y_i - y_j, x \rangle = b_{ij} \}$ and $L_2 = \{x: \langle y_i - y_k, x \rangle =b_{ik}\} = \{ x : \langle y_i-y_j,x \rangle = c^{-1}b_{ik}  \}$. Since $L_1$ and $L_2$ are parallel, we have $\dist(L_1, L_2) = \frac{|b_{ij} - c^{-1}b_{ik}|}{\|y_i - y_j\|}$.
In addition, if $x \in A_{ijk}(t)$, then
\[
\dist(x,L_1) = 0 \quad \text{and} \quad \dist(x, L_2) \leq \frac{R_{\calY} t}{\|y_i - y_k\|}.
\label{eq:dist_condition}
\]
Arguing as in the proof of Lemma \ref{lem:slack_integral}~(ii), one can show that there exists a sufficiently small $t_0$ that depends only on $\calX,\rho,\calY$, and $c_0$ such that $A_{ijk}(t) = \varnothing$ for all $t \in (0,t_0)$.

\medskip
Now, since the Hausdorff measure is translation invariant, we have
\begin{equation}
\calH^{d-1}\big([H_{ij}(0) + t v_{ij}] \setminus H_{ij}(t)\big) \le \sum_{k \ne i,j} \calH^{d-1}(A_{ijk}(t)) \lesssim t, \quad t \in (0,t_0). 
\label{eq: volume}
\end{equation}
For $t \ge t_0$, one may use the crude estimate 
$
\calH^{d-1}\big([H_{ij}(0) + t v_{ij}] \setminus H_{ij}(t)\big) \le \calH^{d-1}(H_{ij}(0)) \le (\diam \calX)^{d-1}
$
and adjust the constant in $\lesssim$ to see that the estimate \eqref{eq: volume} holds for all $t > 0$.

Next,  consider the set $H_{ij}(t) \setminus [H_{ij}(0) + t v_{ij}]$. Each $x \in H_{ij}(t) \setminus [H_{ij}(0) + t v_{ij}]$ satisfies $\langle y_i-y_j,x-tv_{ij} \rangle = b_{ij}$, so one must have $x - t v_{ij} \in C_i(z^0)^c \cap C_j(z^0)^c$. 
This implies that 
$
H_{ij}(t) \setminus [H_{ij}(0) + t v_{ij}]\subset \bigcup_{k \ne i,j} [\tilde H_{ijk}(t) + t v_{ij}],
$
where $\tilde H_{ijk}(t) = \big\{x \in C_k(z^0): x + t v_{ij} \in C_i(z^0), \langle y_i - y_j,x \rangle = b_{ij} \big\}$.
In this case, each $x \in \tilde H_{ijk}(t)$ satisfies 
$-R_{\calY}t \leq \langle y_i - y_k, x \rangle -b_{ik} \leq 0$, 
so that $\tilde H_{ijk}(t) \subset  \big \{x \in \calX: b_{ik} -R_{\calY}t \leq \langle y_i - y_k, x \rangle \le b_{ik}, \langle y_i - y_j, x \rangle = b_{ij} \} =: B_{ijk}(t)$.
Arguing as in the previous case, we have $\calH^{d-1}(B_{ijk}(t)) \lesssim t$. This completes the proof.
\end{proof}

\begin{remark}[Comparison with \cite{altschuler2022asymptotics,delalande2022nearly}]
\label{rem: comparison}
A key estimate in the proofs of Theorem~1.1 in \cite{altschuler2022asymptotics} and Theorem 2.3 in \cite{delalande2022nearly} that concern the asymptotic expansions of the entropic cost is on the integral $\int_{C_i(z^0)} \Delta_{ij}^0(x) \frac{e^{-\Delta_{ij}^0(x)/\varepsilon}}{\sum_{k=1}^N e^{-\Delta_{ik}^0(x)/\varepsilon}} \rho(x) \, dx$. Crucial to their derivations is to use the fact that $\Delta_{ij}(x) \ge 0$ on $C_i(z^0)$ to upper and lower bound the integral. Then, applying the coarea formula and change of variables $t/\varepsilon \to t$ leads to the $O(\varepsilon^2)$ rate. 
In our case, the integrand in \eqref{eq: integrand} need not be nonnegative nor a function of $\Delta_{ij}(x)$, so different arguments are needed. 
\end{remark}

\subsection{Proof of Corollary \ref{cor: clt}}
Let $\hat{q}_{n,i} = \hat{Q}_n(\{ y_i \})$, then $\min_{i} \hat{q}_{n,i} \ge c_0/2$ with probability approaching one. Hence, Corollary \ref{cor: dual norm} yields $\| \hat{T}^{\varepsilon_n}_n - \hat{T}_n^0\|_{(\calC^\alpha)^*} \lesssim \varepsilon_n^{1+\alpha} \vee \varepsilon_n^2 \log^3(1/\varepsilon_n)$. It remains to verify that the central limit theorem \eqref{eq: clt} for $\hat{T}_n^0$ holds under our assumption. To this end, it suffices to verify Assumptions 1 and 2 in \cite{sadhu2023limit}. Assumption 1 in \cite{sadhu2023limit} holds under the current Assumption \ref{asp: compact_convex_regular_density} and the additional assumption made in the statement of the corollary. To verify Assumption 2 in \cite{sadhu2023limit} ($L^1$-Poincar\'{e} inequality for $P$), we first note that it suffices to verify the $L^1$-Poincar\'{e} inequality with the expectation replaced by the median; cf. Lemma 2.1 in \cite{milman2007role}. Recall that the median minimizes the expected absolute deviation. Since $\calX$ is convex, the uniform distribution over $\calX$ satisfies (the median version of) the $L^1$-Poincar\'{e} inequality with constant $K$, say; cf. \cite{bobkov1999isoperimetric}. For any smooth function $f$ on $\R^d$,
\[
\min_{c} \int |f-c| \, dP \le M_\rho \min_c \int_{\calX}|f-c| \, dx  \le   \frac{KM_\rho}{\inf_{x \in \calX}\rho(x)} \int \| \nabla f \| \, dP.
\]
This implies that $P$ satisfies Assumption 2 in \cite{sadhu2023limit}.
\qed

\bibliographystyle{alpha}
\bibliography{ref}

\end{document}